\documentclass{amsart}%
\usepackage{amsfonts}
\usepackage{amsmath}
\usepackage{amssymb}
\usepackage{graphicx}
\usepackage{version}
\usepackage{bm}
\usepackage{enumerate}
\usepackage{epsfig}
\usepackage{caption}%
\usepackage{setspace}

\theoremstyle{plain}
\newtheorem{theorem}{Theorem}[section]

\theoremstyle{definition}

\newtheorem{remark}{Remark}[section]
\newtheorem*{notation*}{Notation}
\numberwithin{equation}{section}
\excludeversion{comment1}

\begin{document}
\title{Towards a More General Type of  Univariate Constrained Interpolation With Fractal Splines}
\author[A.K.B. Chand]{A.K.B. Chand}
\address{Indian Institute of Technology Madras}
\author[P. Viswanathan]{P. Viswanathan}
\address{Australian National University}
\author[K.M. Reddy]{K.M. Reddy}
\address{Indian Institute of Technology Madras}
\maketitle
\begin{abstract}
Recently, in [Electronic  Transaction on Numerical  Analysis, 41 (2014), pp. 420-442] authors introduced a new class of rational cubic  fractal interpolation functions with linear denominators via fractal perturbation of traditional nonrecursive rational cubic splines and investigated their basic shape preserving properties. The main goal of the current article is to embark on  univariate constrained fractal interpolation that is more general than what was considered so far.  To this end, we propose some strategies for selecting the parameters of the rational fractal spline so that the interpolating curves lie strictly above or below a prescribed linear or  a quadratic spline function. Approximation property of the proposed rational cubic fractal spine is broached by using the Peano kernel theorem as an interlude. The paper also provides an illustration of background theory, veined by examples.
\end{abstract}
\textbf{Keywords.} Rational fractal interpolation function; Constrained interpolation; Peano Kernel; Convergence.\\

\textbf{AMS subject Classifications. } 28A80; 65D05; 65D07; 41A29; 41A30; 41A25.

\section{Introductory Remarks}\label{ICAFW2sec1}
Interpolation, which deals with the construction of a function in continuum from its availability in a finite set of points,  is a rather old problem dating back to ancient Babylon and Greece.
In view of its increasing relevance in this age of ever-increasing digitization, it is quite  natural that the subject of interpolation is
receiving more and more attention. Consequently, a multitude of different interpolation schemes are being developed and reported in various places in the literature. However,
in times where all efforts are directed towards construction of smooth functions, the fact that many experimental and natural signals are ``rough" having a dense set of nondifferentiable points or even nowhere differentiable  was easily ignored.

\par
In 1986,  Barnsley \cite{B1} introduced the idea of fractal interpolation function (FIF) aiming at data that require a continuous representation with ``high irregularity''. In words of Barnsley: .....(these functions)  appear ideally suited for approximation of naturally occurring functions
which display some kind of geometrical self-similarity under magnification, for instance,
profiles of mountain ranges, tops of clouds and horizons over
forests, temperatures in
flames as a function of time, electroencephalograph
pen traces and the minute by minute stock market index. In 1989, Barnsley and Harrington \cite{B2} provided a seminal result on differentiability of fractal functions, one which initiated a striking relationship between fractal functions and traditional nonrecursive interpolants, and lead to the creation of a subject which is now called fractal splines.  Since then,  many researchers have contributed to the theory of fractal functions by constructing
various types of FIFs, including  Hermite or spline FIFs (see, for instance, \cite{CV2,NS2}), hidden variable FIFs (see, e.g., \cite{B3,BD2,CK3}), multivariate FIFs (see, e.g., \cite{BD,D}). An operator theoretic formalism for fractal functions which enabled them to find extensive applications in various classical branches of mathematics was introduced and popularized by Navascu\'{e}s \cite{N1,N9}.

\par
In spite of about a quarter century of its first pronouncement, FIFs were extensively investigated only in terms of their algebraic and analytical properties such as fractal dimension, H\"{o}lder exponent, differentiability, integrability, stability, and perturbation errors. Thus, these studies completely ignore one of the important properties
of an interpolation scheme, namely, preserving shape inherent in the data. Constrained control of interpolating curve is a fundamental task and is an important subject that we face in applications including computer aided geometric design, data visualization, image analysis, cartography. However, fractal functions are not well explored in the field of constrained interpolation. Motivated by theoretical and practical needs, the authors have initiated
the study of shape preserving interpolation and approximation using fractal functions; see, for instance, \cite{VC3,VC4}. However,  these researches concern about preserving the three basic shape properties, namely, positivity, monotonicity, and convexity. On the other hand, there are practical situations wherein interpolating curves that lie completely  above or below a prefixed curve, for instance, a polygonal (piecewise linear function) or a quadratic spline are sought-after. The current article can be viewed as  a contribution in this vein. The impetus for the research direction reported here grew out of the works of Qi Duan et al. \cite{QD1,QD2}.

\par
The structure of this article is as follows. We  assemble some relevant definitions and facts concerning fractal interpolation in Section \ref{ICAFW2sec1}. In Section
\ref{ICAFW2sec2}, we briefly recall the rational cubic fractal spline studied in \cite{VC4}. Section \ref{ICAFW2sec3} is devoted to establish convergence of the rational cubic fractal spline. Strategies to select the free parameters involved in the rational cubic FIFs so that they render interpolating curves that lie below or above a prescribed spline (piecewise defined function) are enunciated  in Section \ref{ICAFW2sec4}. In section \ref{ICAFW2sec5}, we provide some numerical examples.


\section{Univariate Fractal Function: An Overview}\label{ICAFW2sec1}

This section targets to  a constellation of  a few rudiments of fractal interpolation theory. For a detailed exposition, the reader may consult the references \cite{B1,N1}. \\
The following notation and terminologies will be used throughout
the article. The set of real numbers will be denoted by $\mathbb{R}$, whilst set of natural numbers by $\mathbb{N}$. For a fixed $N \in \mathbb{N}$, we shall write $\mathbb{N}_N$ for the set of first $N$ natural numbers. Given real numbers $x_1$ and $x_N$ with $x_1<x_N$, we define $\mathcal{C}^k[x_1,x_N]$ to be the space of all real-valued functions on $[x_1,x_N]$ that are $k$-times differentiable with continuous $k$-th derivative.  \\
Let $N >2$ and $\{(x_i, y_i) \in \mathbb{R}^2: i \in \mathbb{N}_N\}$ denote the Cartesian coordinates of a finite set of
points in the Euclidean plane with strictly increasing abscissae. Set $I = [x_1,x_N]$ and $I_i = [x_i,x_{i+1}]$ for $i \in \mathbb{N}_{N-1}$. In fractal interpolation, a continuous function $g: I \to \mathbb{R}$ satisfying $g(x_i)=y_i$ for all $i \in \mathbb{N}_N$ and whose graph $G(g)$ is a fractal (self-referential set) in the sense that $G(g)$ is a union of transformed copies of itself is sought for. There are two approaches for constructing fractal interpolation functions. First method introduced by Barnsley characterizes the graph of FIF  as an attractor of specifically chosen iterated function system, hence  the graph of the fractal interpolant  can be
approximated by the ``chaos game" algorithm \cite{BAV}. In what follows, we supply the second approach, to wit, the construction of a fractal function as the unique fixed point of Read-Bajraktarevi\'{c} operator defined on a suitable function space, which was popularized by Massopust \cite{PM}.\\
 Suppose $L_i: I \rightarrow I_i$, $i \in \mathbb{N}_{N-1}$ be affine maps satisfying
\begin{equation*}\label{ICAFWeq1}
L_i(x_1)=x_i,\ L_i(x_N)=x_{i+1}.
\end{equation*}
For $i\in \mathbb{N}_{N-1}$, let $0<k_i<1$ and
 $F_i: I \times \mathbb{R} \to \mathbb{R}$ be functions that are continuous in first argument and fulfill
 \begin{equation*}\label{ICAFWeq2}
F_i(x_1,y_1)=y_i,\ \ F_i(x_N,y_N)=y_{i+1},~~ \vert
F_i(x,y)-F_i(x,y^*)\vert \leq k_i \vert
 y-y^*\vert.
\end{equation*}
Define
\begin{equation*}
\mathcal{C}^*(I)= \{h \in \mathcal{C}(I): h(x_1)=y_1, h(x_N)=y_N\},~~
\mathcal{C}^{**}(I)=~ \{h \in \mathcal{C}(I): h(x_i)=y_i, i \in \mathbb{N}_N\}.
\end{equation*}
It is valuable to note that both $\mathcal{C}^*(I)$ and $\mathcal{C}^{**}(I)$ are closed (metric) subspaces of the Banach space $\big(\mathcal{C}(I), \|.\|_\infty \big)$. Define
$T: \mathcal{C}^*(I) \to \mathcal{C}^{**}(I) \subset  \mathcal{C}^*(I)$ via
$$ Th (x) = F_i \big(L_i^{-1}(x), h \circ L_i^{-1}(x)\big), ~ x \in I_i, ~ i \in \mathbb{N}_{N-1}.$$
The mapping $T$ is a contraction with contraction factor $k:= \max\{k_i: i \in \mathbb{N}_{N-1}\}$, and the fixed point $g$ of $T$ interpolates the data $\{(x_i,y_i): i \in \mathbb{N}_N\}$. Consequently, $g$ enjoys the functional equation
$$g\big(L_i(x)\big)= F_i\big(x, g(x)\big),~ x\in I,~ i \in \mathbb{N}_{N-1}.$$
By the Banach fixed point theorem, it follows that $g$ can be evaluated by using
$$g= \lim _{k \to \infty} T^k (g_0),$$
where $T^k$ denotes the $k$-fold composition of $T$, and $g_0\in \mathcal{C}^*(I)$ is arbitrary. Further, by the collage theorem, for any $h \in \mathcal{C}^*(I)$
$$\|h-g\|_\infty \le \frac{1}{1-c}\|h-Th\|_\infty.$$
If $w_i: I\times \mathbb{R} \to I_i \times \mathbb{R}$ is defined by $w_i(x,y)=\big(L_i(x),F_i(x,y)\big)$
$\forall~i \in \mathbb{N}_{N-1} $, then
$$G(g)= \cup_{i \in \mathbb{N}_{N-1}} w_i \big (G(g) \big).$$
Next let us point out the well-known fact that the notion of fractal interpolation can be applied to associate  a family of fractal functions
with a prescribed function $f \in \mathcal{C}(I)$. This was  observed originally by Barnsley and explored in detail by Navascu\'{e}s. Consider the maps
\begin{equation*}\label{ICAFWeq4}
L_i(x)=a_ix+c_i,~  F_i(x,y)=\alpha_i y + f \circ L_i(x) -\alpha_i b(x),~
i\in \mathbb{N}_{N-1},
\end{equation*}
where $f\not\equiv b: I \to \mathbb{R}$ is a continuous function interpolating $f$ at the extremes of the interval and $\alpha_i$ are real parameters satisfying $|\alpha_i|<1$, termed scaling factors. The corresponding FIF denoted by
$f^{\alpha}_{\Delta, b}= f^\alpha$ is referred to as $\alpha$-fractal function for $f$ (fractal perturbation of $f$) with scale vector $\alpha$, base function $b$ and  partition $\Delta= \{x_1, x_2, \dots, x_N\}$ of $I$ with increasing abscissae. Note that here the interpolation points are $\{(x_i,f(x_i)):i \in \mathbb{N}_N\}$. The function $f^\alpha$ may be nondifferentiable and its fractal dimension (Minkowski dimension or Hausdorff dimension) depends on the parameter $\alpha \in (-1,1)^{N-1}$. Further, $f^\alpha$  satisfies the functional equation
\begin{equation}\label{ICAFWeq4}
f^\alpha(x)= f(x) + \alpha_i \big(f^\alpha-b\big)(L_i^{-1}(x)), ~x \in I_i,~i \in \mathbb{N}_{N-1}.
\end{equation}
With the notation $|\alpha|_\infty:= \max \{|\alpha_i|: i \in \mathbb{N}_{N-1}\}$, the following inequality points towards  the approximation of $f$ with its fractal perturbation $f^\alpha$.
\begin{equation}\label{ICAFWeq5}
\|f^\alpha-f\|_\infty \le \frac{|\alpha|_\infty}{1-|\alpha|_\infty}\|f-b\|_\infty.
\end{equation}
\section{Rational Cubic Spline FIF Revisited}\label{ICAFW2sec2}
Let $\{(x_i, y_i): i \in \mathbb{N}_N\}$ be a prescribed set of interpolation data with strictly increasing abscissae. Let $d_i$ be the derivative value  at the knot point $x_i$ which is given or estimated with some standard procedures. For $i \in \mathbb{N}_{N-1}$, denote by $h_i$, the local mesh spacing  $x_{i+1}-x_i$, and let  $r_i$ and $t_i$ be positive parameters referred to as shape parameters for obvious reasons. A $\mathcal{C}^1$-continuous  rational cubic spline with linear denominator was introduced in \cite{QD2} as follows.
\begin{equation*} \label{11a}
f\big(L_i(x)\big)= \frac{(1-\theta)^3 r_i y_i+\theta (1-\theta)^2
V_i+ \theta^2 (1-\theta)W_i +\theta^3 t_i y_{i+1}}{(1-\theta)r_i+
\theta t_i},\; i \in \mathbb{N}_{N-1},
\end{equation*}
where
\begin{equation*} \label{12}
V_i =(2r_i + t_i) y_i+ r_i h_i d_i,\;\; W_i =(r_i + 2t_i) y_{i+1}-
t_i h_i d_{i+1}, \;\;\theta=\frac{x-x_1}{x_N-x_1}.
\end{equation*}
In  reference \cite{VC4}, the authors  observe that for the construction of fractal perturbation $f^\alpha$ of the rational spline  $f$ involving shape parameters, it is more advantageous to work with a family of base functions $\{b_i: i \in \mathbb{N}_{N-1}\}$
instead of a single base function used in the traditional setting (see Eq. (\ref{ICAFWeq4})). Furthermore, for the fractal perturbation $f^\alpha$ to be $\mathcal{C}^1$- continuous, it suffices to choose the scaling parameters $\alpha_i$ such that $|\alpha_i|<a_i$ and base functions $\{b_i: i \in \mathbb{N}_{N-1}\}$ so that each $b_i$ agrees with $f$ at the extremes of the interpolation interval up to the first derivative. Among various possibilities, the following choice is motivated by the simplicity it offers for the final expression of the desired fractal analogue of $f$.
\begin{equation}\label{13}
b_i (x) = \frac{B_{1i} (1-\theta)^3 + B_{2i} \theta (1-\theta)^2 +
B_{3i} \theta^2 (1-\theta)+ B_{4i} \theta^3}{(1-\theta)r_i+ \theta
t_i},
\end{equation}
where the coefficients $B_{1i}$, $B_{2i}$, $B_{3i}$, and
$B_{4i}$ are prescribed as
\begin{equation*}\label{14}
\begin{split}
 B_{1i} &=  r_i y_1,\;\; B_{2i} =(2 r_i +t_i) y_1 +r_i d_1
 (x_N-x_1),\\
 B_{3i}& =( r_i + 2t_i) y_N-t_i
d_N(x_N-x_1),\;\;B_{4i} = t_i y_N.
\end{split}
\end{equation*}
Therefore, in view of (\ref{ICAFWeq4}), the desired $\mathcal{C}^1$-continuous rational cubic fractal spline is given by the expression
\begin{eqnarray}\label{15}
f^\alpha\big(L_i (x)\big)&=& \alpha_i f^\alpha(x)+ \frac{P_i(x)}{Q_i(x)},\; \text{where} \nonumber\\
P_i(x)=P_i^*(\theta) &=&(y_i-\alpha_i y_1) r_i (1-\theta)^3+(y_{i+1}-\alpha_i
y_N)t_i\theta^3+\big\{(2 r_i + t_i) y_i\nonumber \\&&+ r_i h_i d_i
-\alpha_i[(2 r_i + t_i)y_1+ r_i (x_N-x_1)
d_1]\big\}\theta(1-\theta)^2+\big\{(r_i+ 2t_i)y_{i+1} \nonumber \\&&-t_i
h_i d_{i+1}  -\alpha_i[(r_i +2t_i)y_N-t_i(x_N - x_1)d_N]\big\}\theta^2(1-\theta), \nonumber\\
Q_i(x)=Q_i^*(\theta)&=&(1-\theta) r_i+ \theta t_i,~~ i \in \mathbb{N}_{N-1},~ \theta=
\frac{x-x_1}{x_N-x_1}.
\end{eqnarray}
\begin{remark}
A rational cubic spline with linear denominator which is based only on the function values is introduced in \cite{QD1}, pointing out that
in some manufacturing processes the derivatives are difficult to obtain. However, as far as we know, it differs from (\ref{11a}) and its construction only in the following way. Consider the given set of data points $\Delta_1=\{(x_i,y_i): i \in \mathbb{N}_{N+1}\}$ and the subset $\Delta_2:=\{(x_i,y_i): i \in \mathbb{N}_N\}$ of interpolation points. Treating $d_i$ to be equal to the chord slope $\Delta_i= \frac{y_{i+1}-y_i}{h_i}$ for $i \in \mathbb{N}_{N}$, the rational cubic spline involving only the function values discussed in detail in \cite{QD1}  can be deduced at once from (\ref{11a}). Consequently, on similar lines, the following expression for the fractal rational cubic splines involving only  function values can be obtained.
\begin{eqnarray}\label{22}
\hat{f}^\alpha(L_i(x))&=&\alpha_i \hat{f}^\alpha(x)+ \frac{\hat{P_i}(x)}{\hat{Q_i}(x)},\nonumber\\
\hat{P_i}(x) &=& (y_i-\alpha_iy_1)r_i(1-\theta)^3+ (y_{i+1}-\alpha_{i}y_{N})t_{i}\theta^3+\nonumber\\
&&\{(r_i +t_i)y_i+r_i y_{i+1}-\alpha_i
[(2r_i+ t_i) y_1+
r_i(x_N-x_1)\Delta_1]\}\theta(1-\theta)^2+\nonumber\\
&& \{(2 t_i+r_i)y_{i+1}- h_i t_i
\Delta_{i+1}-\alpha_i[(2t_i
 + r_i)y_N-t_i (x_N-x_1)\Delta_N]\} \theta^2(1-\theta),\nonumber\\
\hat{Q_i}(x)&=&r_i (1-\theta) +  t_i \theta,~ i \in \mathbb{N}_{N-1},~ \theta=
\frac{x-x_1}{x_N-x_1}.
\end{eqnarray}
Hence, in principle, the analysis on the fractal spline structure (\ref{15}) set about to do in the subsequent sections can be easily modified and adapted to the rational cubic spline FIF given in Eq. (\ref{22}).
\end{remark}
\section{Convergence Analysis}\label{ICAFW2sec3}
The rate at which an interpolant $g$ approaches $\Phi$, the unknown function generating the data $\{(x_i,y_i): i \in \mathbb{N}_N\}$, is perhaps the most important factor deciding the efficacy of an interpolation scheme. In \cite{VC4}, it has been established that the rational cubic fractal spline $f^\alpha$ converges to $\Phi \in \mathcal{C}^2(I)$ with respect to the $\mathcal{C}^2$-norm. Since order of continuity of $f^\alpha$ is $\mathcal{C}^1$, it is natural to ask whether $f^\alpha$ possesses uniform  convergence as the norm of the partition approaches zero, if the data generating function is assumed to have a reduced order of continuity, namely $\mathcal{C}^1$. In this section, we answer this in the affirmative. \\
Note that the implicit and recursive nature inherent in  the description of the  fractal spline $f^\alpha$ make it impossible or at least difficult to apply  the standard tools for  error analysis available in the classical interpolation theory for establishing its convergence. On the other hand, we can base our analysis on the  trustworthy triangle inequality
\begin{equation}\label{new1}
\|\Phi-f^\alpha\|_\infty \le \|\Phi-f\|_\infty + \|f-f^\alpha\|_\infty.
\end{equation}
Following (\ref{ICAFWeq5}), it can be deduced  at once that
\begin{equation}\label{new2}
\|f-f^\alpha\|_\infty \le \frac{|\alpha|_\infty}{1-\alpha|_\infty} \big(\|f\|_\infty + \max_{i \in \mathbb{N}_{N-1}} \|b_i\|_\infty\big).
\end{equation}
It can be read from \cite{VC4} that
\begin{eqnarray*}
\|f\|_\infty &\le& |y|_\infty + \frac{h}{4} |d|_\infty,\\
\|b_i\|_\infty &\le& \max \{|y_1|, |y_N|\}+ \frac{|I|}{4}\max \{|d_1|,|d_N|\},
\end{eqnarray*}
where $|y|_\infty:= \{|y_i|: i \in \mathbb{N}_{N}\}$, $|d|_\infty:= \{|d_i|: i \in \mathbb{N}_{N}\}$, $h:=\{h_i: i \in \mathbb{N}_{N-1}\}$, and $|I|:=x_N-x_1$. Observing that $|\alpha_i|<a_i=\frac{h_i}{|I|}$ for all $i \in \mathbb{N}_{N-1}$,  we obtain
$$ \frac{|\alpha|_\infty}{1-\alpha|_\infty} \le \frac{h}{|I|-h}.$$
Hence from (\ref{new2}),  the perturbation error obeys $$\|f-f^\alpha\|_\infty = O(h), ~\text{as}~ h \to 0.$$
As an interlude to our analysis, we have the following theorem which deals with the first summand of the inequality (\ref{new1}), namely, the interpolation error of the traditional rational cubic spline $f$ (cf. Eq. (\ref{11a})) with respect to the uniform norm.
\begin{theorem}\label{errorthm1}
Let $\{(x_i,y_i): i \in \mathbb{N}_N\}$ be a prescribed set of interpolation data generated by a function $\Phi \in \mathcal{C}^1(I)$ and $f$ be the corresponding traditional nonrecursive rational cubic spline with linear denominator. Assume that the derivatives at the knots are given or estimated by some linear approximation methods (for instance, arithmetic mean method). Then the local error of the interpolation is given by
$$|\Phi(x)-f(x)| \le ~h_i ~c_i\|\Phi'\|_\infty,~ \text{for}~ x \in I_i,$$ where for the local variable $\varphi:=\frac{x-x_i}{h_i}$, the constant $c_i$ is given by
$$c_i =\max_{0 \le \phi \le 1} \frac{r_i\varphi(1-\varphi)^2(1+2\varphi)+t_i \varphi^2 (1-\varphi)(3-2\varphi)}{r_i(1-\varphi)+t_i \varphi}. $$
\end{theorem}
\begin{proof}
Consider the error function $E(\Phi;x)=\Phi(x)-f(x)$ as a linear functional which operates on $\Phi$. Direct calculations confirm that the linear functional annihilates polynomials of degree strictly less than one.  By the Peano kernel theorem \cite{Pow}:
\begin{equation}\label{e1}
 L[\Phi]=E(\Phi;x)=
\Phi(x)-f(x)=\int_{x_i}^{x_{i+1}}\Phi'(\tau)L_{x}[(x-\tau)^0_{+}
]~d\tau,
\end{equation}
where the kernel function is given by
\begin{displaymath}
   L_{x}[(x-\tau)^0_+] = \left\{
     \begin{array}{lr}
       r(\tau,x) ~~ \text{if}~~ x_{i}<\tau<x,\\
       s(\tau,x) ~~ \text{if}~~ x<\tau<x_{i+1}.
     \end{array}
   \right.
\end{displaymath}
Here the notation $L_{x}$ is used to emphasize that the functional
$L$ is applied to the  truncated power
function
\begin{displaymath}
   (x-\tau)^n_+ := \left\{
     \begin{array}{lr}
       (x-\tau)^n ~~ \text{if}~~ \tau<x,\\
       ~~~0~~~~~~ ~~~~~ \text{if}~~ \tau >x,
     \end{array}
   \right.
\end{displaymath}
considered as a
function of $x$. Denoting $\varphi:= \frac{x-x_i}{h_i}$, a rigorous calculation yields
\begin{eqnarray*}
r (\tau,x) &=&
\frac{r_i (1-\varphi) (1-\varphi^2)+t_i\varphi (1-\varphi)^2}{(1-\varphi)r_i+\varphi t_i},\label{18b}\\
s (\tau,x) &=& -\frac{r_i \varphi^2 (1-\varphi)+t_i\varphi^2 (2-\varphi)}{(1-\varphi)r_i+\varphi t_i}.
\end{eqnarray*}
Using the above expression for the kernel function in (\ref{e1}), we obtain the following
chain of relations:
\begin{eqnarray*}
|\Phi(x)-C(x)|  &= & \Big |\int_{x_i}^{x_{i+1}}
\Phi'(\tau)L_{x}[(x-\tau)^0_{+}]~d\tau \Big|,\\
& \le & \|\Phi'\| \int_{x_i}^{x_{i+1}} \big|L_x [(x-\tau)^0_{+}
]\big|~d\tau, \\
& \le &\|\Phi'\| \Big\{\int_{x_i}^{x} r(\tau, x)~d\tau
-\int_{x}^{x_{i+1}} s(\tau,x)~d \tau \Big \},\\
&=&\|\Phi'\|\Big\{ \frac{r_i (1-\varphi)(1-\varphi^2)+t_i\varphi(1-\varphi)^2}{(1-\varphi)r_i+\varphi t_i}(x-x_i)\\&&+\frac{r_i \varphi^2(1-\varphi)+t_i\varphi^2(2-\varphi)}{(1-\varphi)r_i+\varphi t_i}(x_{i+1}-x)\Big\},\\
&=& \|\Phi'\|h_i c_i,
\end{eqnarray*}
where $\|.\|$ denotes the uniform norm on the subinterval $I_i=[x_i,x_{i+1}]$. This offers the promised result.
\end{proof}
Moving now to the crux of this section, we have the following theorem, whose proof  is immediate from the foregoing discussion and the theorem.
\begin{theorem}
Let $\{(x_i,y_i): i \in \mathbb{N}_N\}$ be a prescribed set of interpolation data generated by a function $\Phi \in \mathcal{C}^1(I)$. Assume further
that the derivatives at the knots are obtained by some linear approximation methods. Let  $f^\alpha$
be the corresponding rational cubic spline FIF. Then
\begin{eqnarray*}
\|\Phi-f^\alpha\|_\infty &\le& \frac{|\alpha|_\infty}{1-|\alpha|_\infty}\Big[|y|_\infty+ \max \{|y_1|, |y_N|\}+  \frac{1}{4}\big(h |d|_\infty+|I| \max\{|d_1|,|d_N|\}\big)\Big]\\ &&+ c h\|\Phi'\|_\infty,
\end{eqnarray*}
where $c:=\max\{c_i: i \in \mathbb{N}_{N-1}\}.$
\end{theorem}
The following remark highlights the important fact to which the preceding theorem sheds light.
\begin{remark} If $\Phi \in\mathcal{C}^1(I)$ is the data generating function and $f^\alpha$ is the rational cubic fractal spline corresponding to the data set, then
$$\|\Phi-f^\alpha\|_\infty =O(h), ~\text{as}~ h \to 0.$$ Ergo, $f^\alpha$ converges uniformly to the original function, as the norm
 of the partition tends to zero.
\end{remark}
\begin{remark}
 Let us point out, at least for the sake of good bookkeeping, that the application of triangle inequality does not imply that the error in approximating a continuously differentiable function $\Phi$ with the fractal spline $f^\alpha$  is always greater than or equal to the error in approximation by its classical counterpart $f$.
We employed the triangle inequality just to demonstrate that the fractal spline $f^\alpha$ has the same order of convergence as that of the classical rational spline $f$. Finding appropriate $\alpha$ for which $f^\alpha$ is close to $\Phi$ is an entirely different problem, which can be shown to be a constrained convex optimization problem with the aid of collage theorem \cite{VC3}. There is no reason to believe that the solution would always be $\alpha=0$, which corresponds to the classical interpolant $f$.
\end{remark}
\section{Constrained Interpolation with Rational Cubic Spline FIF} \label{ICAFW2sec4}
This section features a systematic discussion of selection of parameters involved in the rational cubic spline FIF,  focusing on its applicability in a more general constrained interpolation problem. To be precise, given a data set $\{(x_i,y_i): i\in \mathbb{N}_N\}$ and a function $p$ (piecewise linear or  piecewise quadratic with joints at the knots  $x_i$) satisfying $ y_i \ge p(x_i)$, the problem is to construct a rational cubic spline FIF $f^\alpha$ such that $f^\alpha(x) \ge p(x)$ for all $x \in I$. The reader is bound to have noticed that the values of $f^\alpha$ in the subinterval $[x_i,x_{i+1}]$ depends on its values on other subintervals as well. Hence, obtaining conditions for which $f^\alpha$ lies above a piecewise defined function is relatively harder than that in the corresponding classical counterpart. However, this can be remedied by connecting $f^\alpha$ with its classical counterpart $f^0=f$, and performing  the analysis on $f$ rather than on $f^\alpha$ itself. Let us commence by noting that \begin{eqnarray*}
 f^\alpha(x)-p(x) &=&f^\alpha(x)-f(x)+ f(x)-p(x)\\
 &\geq& \frac{|\alpha|_\infty}{|\alpha|_\infty-1}M + f(x)-p(x),
\end{eqnarray*}
where $$M=|y|_\infty+ \max \{|y_1|, |y_N|\}+  \frac{1}{4}\big(h |d|_\infty+|I| \max\{|d_1|,|d_N|\}\big).$$
For brevity, we denote
$$K= \frac{|\alpha|_\infty}{|\alpha|_\infty-1}M.$$
It follows that for $f^\alpha(x)-p(x)\ge 0$, it suffices to make
$$f(x) \ge p(x)-K.$$
Recall that on $[x_i, x_{i+1}]$, $f(x)=\frac{R_i(x)}{S_i(x)}$, where with local variable $\varphi:=\frac{x-x_i}{h_i}$,
\begin{eqnarray*}
R_i(x)& =& (1-\varphi)^3 r_i y_i+ \varphi (1-\varphi)^2 [(2r_i+t_i)y_i+r_ih_id_i]+ \varphi^2 (1-\varphi) [(r_i+2t_i) y_{i+1}\\&&-t_ih_id_{i+1}]+\varphi^3 t_i y_{i+1},\\
S_i(x)&=& (1-\varphi) r_i + \varphi t_i.
\end{eqnarray*}
Let $p$ be a piecewise linear function with joints at $x_i$, $i \in \mathbb{N}_N$. If $p_i=p(x_i)$ and $p_{i+1}=p(x_{i+1})$, then we have
$$p(x)= p_i (1-\varphi) + p_{i+1} \varphi, ~ x \in I_i=[x_i,x_{i+1}].$$
Therefore $f(x) \ge p(x)-K$ for all $x \in I$ is satisfied if
\begin{equation}\label{eqnconstr1}
R_i(x)-[p_i (1-\varphi) + p_{i+1} \varphi-K] S_i(x) \ge 0, ~ \text{for}~ x \in I_i, ~i \in \mathbb{N}_{N-1}.
\end{equation}
By using the technique of degree elevation, we assert  that
\begin{eqnarray*}
[p_i (1-\varphi) + p_{i+1} \varphi] S_i(x)&=& r_i p_i (1-\varphi)^3+ [r_i(p_i+p_{i+1})+p_it_i] \varphi(1-\varphi)^2\\&&+[r_i p_{i+1}+t_i(p_i+p_{i+1})] \varphi^2(1-\phi)+
\varphi^3 t_i p_{i+1},\\
S_i(x)=(1-\varphi)r_i + \varphi t_i &=& r_i(1-\varphi)^3 + (2r_i+t_i)(1-\varphi)^2 \varphi+ (r_i+2t_i)\varphi^2 (1-\varphi)\\&&+ \varphi^3 t_i.
\end{eqnarray*}
In view of the aforementioned pair of equations, the condition (\ref{eqnconstr1}) can be recast as a problem of positivity of a cubic polynomial, namely,
 \begin{eqnarray*}
 &&r_i(y_i-p_i+K) (1-\varphi)^3 + \big[r_i(2y_i+h_id_i-p_i-p_{i+1}+2K)+t_i(y_i-p_i+K)\big]\varphi(1-\varphi)^2\\&&+ \big[r_i(y_{i+1}-p_{i+1}+K)+t_i(2y_{i+1}-h_id_{i+1}-p_i-p_{i+1}+2K)\big] \varphi^2 (1-\varphi)\\&&+ t_i (y_{i+1}-p_{i+1}+K) \varphi^3 \ge 0, ~\text{for all}~i \in \mathbb{N}_{N-1}.
 \end{eqnarray*}
 Conditions for positivity (nonnegativity) of a cubic polynomial is well-studied in the literature (see, for instance, \cite{SH2}). However, to keep our analysis simple enough, we shall impose conditions on parameters so that each coefficient of the cubic polynomial appearing in the above inequality is nonnegative.
 Noting that $y_i \ge p_i$ for all $i \in \mathbb{N}_N$ and $M\ge0$, we obtain
 \begin{eqnarray*}
 y_j-p_j+ K \ge 0 &\Longleftrightarrow& |\alpha|_\infty \le \frac {y_j-p_j}{y_j-p_j+M},~ \text{for}~j=i,i+1.
 \end{eqnarray*}
 The  discussion we had until now is epitomized in the following theorem.
 \begin{theorem}\label{abovelinecon}
 Suppose that a data set $\{(x_i,y_i): i \in \mathbb{N}_N\}$, where $y_i \ge p_i=p(x_i)$ and $p$ is a piecewise linear function  with joints at knots $x_i$ is prescribed. Then a sufficient condition for the rational cubic spline FIF $f^\alpha$ to lie above $p$ is that the parameters satisfy the following inequalities:
 \begin{enumerate}[(i)]
 \item $|\alpha_i|<a_i$ for $i \in \mathbb{N}_{N-1}$, and $|\alpha|_\infty \le \min \big\{\dfrac {y_i-p_i}{y_i-p_i+M}: i \in \mathbb{N}_N\big\}$,\\
 \item $r_i(2y_i-p_{i+1}-p_i+h_id_i+2K) +t_i (y_i-p_i+K)\ge 0,~ i \in \mathbb{N}_{N-1}$,\\
 \item $r_i(y_{i+1}-p_{i+1}+K)+t_i(2 y_{i+1}-p_{i+1}-p_i-h_id_{i+1}+2K) \ge 0, ~ i \in \mathbb{N}_{N-1}$,
 \end{enumerate}
 where
 \begin{eqnarray*}
 M&=&|y|_\infty+ \max \{|y_1|, |y_N|\}+  \frac{1}{4}\big(h |d|_\infty+|I| \max\{|d_1|,|d_N|\}\big),\\
 K&=& \frac{|\alpha|_\infty}{|\alpha|_\infty-1}M.
 \end{eqnarray*}
\end{theorem}
A couple of  remarks that supplement the stated theorem are in order.
\begin{remark}
 On similar lines, conditions for $f^\alpha$ to lie below a piecewise linear function $p*$ with joints at $x_i$ satisfying $p^*(x_i) \le y_i$
 can be derived. By coupling these conditions with those prescribed in Theorem \ref{abovelinecon}, we obtain strategies for selecting the scaling factors and shape parameters so that  $p^*(x) \le f^\alpha(x) \le p(x)$ for all $x \in I$. To keep the article at a reasonable length, we avoid the computational details here.
 \end{remark}
 \begin{remark} \label{exsrem}
 Having selected the scaling parameters according to the prescription in  Theorem \ref{abovelinecon}, the existence of  rational cubic fractal spline $f^\alpha$ that lie above a piecewise linear function $p$  points naturally to the existence of positive shape parameters $r_i$ and $t_i$ satisfying the  inequalities therein. We shall hint on the existence in the following, see also \cite{QD2}. Let us rewrite the inequalities involving $r_i$ and $t_i$ (see items (ii)-(iii) in Theorem \ref{abovelinecon}) as follows.
 \begin{eqnarray}\label{exstineq}
 A_i+ \lambda_i B_i \ge 0,~~~~
C_i + \lambda_i D_i \ge 0,
 \end{eqnarray}
 where $\lambda_i=\frac{t_i}{r_i}$, $A_i=2y_i-p_{i+1}-p_i+h_id_i+2K$, $B_i=y_i-p_i+K$, $C_i=y_{i+1}-p_{i+1}+K$, and $D_i=2 y_{i+1}-p_{i+1}-p_i-h_id_{i+1}+2K$. If $p_i<y_i$ for all $i \in \mathbb{N}_N$ and the scaling factors are chosen such that $|\alpha|_\infty < \min \big\{\frac{y_i-p_i}{y_i-p_i+M}: i \in \mathbb{N}_N\}$ (that is, $B_i>0$, $C_i>0$), then a positive $\lambda_i$ (that is,  a set of positive $r_i$ and $t_i$) satisfying the inequalities in (\ref{exstineq}) exists except when $A_i>0$, $D_i<0$, and $A_iD_i>C_iB_i$.
 \end{remark}
 \begin{remark}
  If we replace $d_i$ with the slopes $\Delta_i$, then the coefficient of $r_i$ in item (ii) of Theorem \ref{abovelinecon} reduces to $y_i-p_i+y_{i+1}-p_{i+1}+2K$. Hence, in view of item (i), the  condition in item (ii) is trivially satisfied. Consequently, we obtain the following conditions for the rational cubic spline FIF $\hat{f}^\alpha$ (cf. (\ref{22})) to lie above a piecewise linear function $p$.
 \begin{eqnarray*}
  &&|\alpha_i|<a_i, ~i \in \mathbb{N}_{N-1},~~ |\alpha|_\infty \le \min \big\{\dfrac {y_i-p_i}{y_i-p_i+M}: i \in \mathbb{N}_N\big\},\\
  &&r_i(y_{i+1}-p_{i+1}+K)+t_i(2 y_{i+1}-p_{i+1}-p_i-h_i\Delta_{i+1}+2K) \ge 0, ~ i \in \mathbb{N}_{N-1}.
  \end{eqnarray*}
 \end{remark}
\noindent  Next let us consider $p$ to be a piecewise defined quadratic polynomial with joints at $x_i$, $i \in \mathbb{N}_N$. If $p_i=p(x_i)$,  $p_{i+1}=p(x_{i+1})$, and $p'(x_i)=p_i'$, then we have
$$p(x)= p_i (1-\varphi)^2 + (2p_i+p_i'h_i) \varphi(1-\varphi)+  p_{i+1} \varphi^2, ~ x \in I_i=[x_i,x_{i+1}].$$
Analysis similar to the case of piecewise linear $p$ establishes the following theorem.
\begin{theorem}\label{fifabvqf}
Suppose that a data set $\{(x_i,y_i): i \in \mathbb{N}_N\}$, where $y_i \ge p_i=p(x_i)$ and $p$ is  a piecewise quadratic  function  with knots at $x_i$ is given. Then a sufficient condition for the rational cubic spline FIF $f^\alpha$ to lie above $p$ is that the parameters satisfy the following inequalities:
 \begin{enumerate}[(i)]
 \item $|\alpha_i|<a_i$ for $i \in \mathbb{N}_{N-1}$, and $|\alpha|_\infty \le \min \big\{\dfrac {y_i-p_i}{y_i-p_i+M}: i \in \mathbb{N}_N\big\}$,\\
 \item $r_i(2y_i-2p_i+h_id_i-h_ip_i'+2K) +t_i (y_i-p_i+K)\ge 0,~ i \in \mathbb{N}_{N-1}$,\\
 \item $r_i(y_{i+1}-p_{i+1}+K)+t_i(2 y_{i+1}-2p_i-h_id_{i+1}-h_ip_i'+2K) \ge 0, ~ i \in \mathbb{N}_{N-1}$,
 \end{enumerate}
 where
 \begin{eqnarray*}
 M&=&|y|_\infty+ \max \{|y_1|, |y_N|\}+  \frac{1}{4}\big(h |d|_\infty+|I| \max\{|d_1|,|d_N|\}\big),\\
 K&=& \frac{|\alpha|_\infty}{|\alpha|_\infty-1}M.
 \end{eqnarray*}
\end{theorem}
 \noindent Note that if $M$ is very large, then admissible value for $|\alpha|_\infty$ will be close to zero. Consequently, the fractal perturbation $f^\alpha$ will ``almost" coincide with $f$. This is an unpleasant situation as far as construction of a  $\mathcal{C}^1$-continuous constrained interpolant with ``fractality" in its derivative is concerned, owing to the fact that larger
the value of $|\alpha|_\infty$ (with respect to the interpolation step), more pronounced is the irregularity in the derivative of the smooth fractal interpolant. Furthermore, in light of Remark \ref{exsrem}, it can be deduced that there are some cases in which the constrained interpolation cannot be solved with the rational cubic spline FIF $f^\alpha$. These bring us motivating influence to seek an alternative approach. We shall now provide a quite brief deliberation of this alternative approach for the constrained interpolation. The following is the key theorem.
\begin{theorem}\label{RRAFIFbelthm}
Let $f \in \mathcal{C}^1(I)$ and $\Delta:=\{x_1,x_2,\dots, x_N\}$ be a partition of $I$ satisfying $x_1<x_2<\dots<x_N$. Let
the scaling factors be chosen such that $\alpha_i \in (0, a_i)$ and the base functions be chosen fulfilling  the conditions $b_i^{(j)}(x_1)=f^{(j)}(x_1)$, $b_i^{(j)}(x_N)=f^{(j)}(x_N)$ for $j=0,1$ and  $b_i(x) \le f(x)$ for all $i \in \mathbb{N}_{N-1}$, $x \in I$. Then the corresponding fractal function $f^\alpha\in \mathcal{C}^1(I)$ and $f^\alpha(x) \ge f(x)$ for all $x \in I$.
\end{theorem}
\begin{proof}
Recall that  fractal function $f^\alpha$ corresponding to $f$ satisfies the functional equation
\begin{equation}\label{RRAFIFeqna}
f^\alpha\big(L_i(x)\big)=f\big(L_i(x) \big) +\alpha_i[f^\alpha(x)-b_i(x)],~ x \in I.
\end{equation}
For $|\alpha_i|<a_i$ and base functions coinciding with $f$ at the extremes of the interval up to the first derivative, it follows  \cite{VC4} that $f^\alpha \in \mathcal{C}(I)$. \\Further, $f^\alpha(x_i)=f(x_i)$, i.e., $(f^\alpha-f)(x_i) \ge 0$ for all $i \in \mathbb{N}_N$ (in fact, the equality).
Note that  $I=\underset{i \in \mathbb{N}_{N-1}}\cup L_i(I)$ and $f^\alpha$ is constructed using an iterated scheme
according to \eqref{RRAFIFeqna}. Whence,  to prove $(f^\alpha-f)(x)\ge 0$ for all $x \in I$ it is enough to prove that $(f^\alpha-f)\ge 0$ holds good at the points on $I$
obtained at $(i+1)$-th iteration whenever $(f^\alpha-f)\ge 0$ is satisfied for the points on $I$ at $i$-th iteration. The aforementioned condition is equivalent to $(f^\alpha-f)\big(L_i(x)\big) \ge 0$ for all $i \in \mathbb{N}_{N-1}$ whenever
 $(f^\alpha-f)(x) \ge 0$. We may rewrite \eqref{RRAFIFeqna} in the form
 \begin{equation}\label{RRAFIFeqnb}
 \begin{split}
 (f^\alpha-f) \big (L_i(x)\big) =&~\alpha_i(f^\alpha-b_i)(x),\\
 =&~\alpha_i(f^\alpha-f)(x)+\alpha_i(f-b_i)(x).
 \end{split}
 \end{equation}
 Assume that the scaling factors are chosen such that $\alpha_i\ge 0$ for all $i \in \mathbb{N}_{N-1}$. By our assumption $(f^\alpha-f)(x) \ge0$, hence an appeal to \eqref{RRAFIFeqnb} reveals that $(f^\alpha-f) \big (L_i(x)\big) \ge 0$ is satisfied whenever $b_i(x) \le f(x)$ for all $x \in I$ and $i \in \mathbb{N}_{N-1}$, offering the proof.
\end{proof}
The foregoing theorem demonstrates that if a constrained interpolation problem is solved with a classical rational interpolant $f$, then the perturbation process can be so designed that the corresponding fractal spline $f^\alpha$ also provides a solution to the same constrained interpolation problem. For instance, let $f$ be the rational cubic spline with linear denominator lying above a piecewise linear function $p$. Choose $\alpha_i \in (0, a_i)$ and the base functions $b_i= f- \hat{b}_i$, where $\hat{b_i}$ are positive functions satisfying $\hat{b}_i^{(j)}(x_1)=0$, $\hat{b}_i^{(j)}(x_N)=0$ for $i \in \mathbb{N}_{N-1}$, $j=0,1$. Since positivity preserving Hermite polynomial and rational spline interpolants are well-studied in the literature, it is not hard to find a family of functions $\{\hat{b}_i:i \in \mathbb{N}_{N-1}\}$  satisfying
  the aforementioned conditions. Note that cubic Hermite interpolants may not serve as candidates for $\hat{b_i}$, as they may reduce to zero functions owing to the imposed conditions. The fractal function $f^\alpha$ corresponding to $f$ with these choices of scaling parameters
  and base functions will lie above $p$. Thus, by dint of a more careful choice on base functions $b_i$, the conditions on $\alpha_i$ can be made relatively weaker, to wit,  $\alpha_i \in (0, a_i)$ for all $i \in \mathbb{N}_{N-1}$.
\section{Numerical Examples}\label{ICAFW2sec5}
In this section, the developed conditions on parameters are leveraged to obtain rational cubic spline FIFs that lie above a prescribed piecewise linear function (polygonal) or quadratic function. Consider the data set $\{(x_i,y_i,d_i), i=1,2,3,4,5\}=\{(0,18, -4.02), (3,10, -1.31), (7,12, -0.36),\\ (10,9, 0.2), (15,20, 4.2)\}$. Note that the prescribed data set lies above the piecewise linear function $p$ with nodes at $\{(0,12),(3,4),(7,10),(10,4)\}$ given by
$$
   p(x) := \left\{
     \begin{array}{lr}
       -\frac{8}{3} x +12 ~~ \text{if}~~ 0\le x \le 3,\\
       \frac{3}{2}x-\frac{1}{2}~~~~ ~~ \text{if}~~ 3 \le x \le 7,\\
       -2x + 24 ~~\text{if}~~7 \le x \le 10,\\
       \frac{7}{5} x -10~~~~~\text{if}~~10 \le x \le 15.
     \end{array}
   \right.
$$
Let the derivatives at knot points be  $\{-4.02,-1.31,-0.36,0.2,4.2\}$. Suppose that, due to some reasons, perhaps for a valid physical interpretation of the underlying process, a constrained  interpolant lying above the following piecewise linear function $p$ is required.
For brevity, let us represent  the parameters,
namely, the  scaling factors $\alpha_i$ satisfying $|\alpha_{i}|<a_i<1$, and the positive shape parameters $r_{i}$ and $t_{i}$ as vectors denoted by $\alpha$, $r$ and $t$ in the four dimensional Euclidean space $\mathbb{R}^4$. The details of the scaling vectors and shape parameters used in the construction of  constrained rational FIFs   are provided in Table \ref{table:Data} for a quick reference.
\begin{center}
\begin{table}[h!]
\caption{Scaling vectors  and shape parameters used in the construction of   rational cubic FIFs.}\label{table:Data}
\begin{center}
\begin{tabular}{|l | l|l|}
\hline
Figs.& Scaling Vectors $\alpha$ \hspace{1cm}&Shape Parameter Vectors $r$, $t$ \hspace{1cm} \\ \hline
Fig. \ref{figabvl}(a)& $\left(0.010,  0.020,  0.030, 0.333\right)$ &  $\left(1, 1, 1,  1 \right)$,  $\left(
 3.35, 1, 1, 1\right)$ \\

Fig. \ref{figabvl}(b) & $\left(0.027,  0.027, 0.027, 0.024\right)$ & $\left(
1, 1, 1, 1\right)$, $\left(3.35, 1, 1, 1\right)$ \\
Fig. \ref{figabvl}(c)&  $\left(0, 0, 0, 0\right)$ & $\left(
1, 1, 1, 1\right)$, $\left(3.35, 1, 1, 1\right)$ \\

Fig. 2(a)& $\left(0.012,  0.013, 0.040, 0.005\right)$ & $\left(
9, 1, 0.01, 11 \right)$, $\left(10, 200, 0.0001, 8\right)$\\
Fig. 2(b)& $\left(0.040, 0.039, 0.025,  0.034\right)$&   $\left(9,  1,  0.01, 11\right)$, $\left(10, 200, 0.0001, 8 \right)$\\
Fig. 2(c)&$\left(0.012,  0.013,  0.040,  0.005 \right)$ &  $\left(19, 11, 0.001, 10\right)$, $\left(12, 210, 0.00001, 7  \right)$ \\
\hline
    \end{tabular}
    \end{center}
\end{table}
\end{center}
By selecting the scaling and  shape parameters according to the conditions prescribed in Theorem \ref{abovelinecon} (see Table \ref{table:Data}), a rational cubic spline FIF (cf. Eq. (\ref{15})) lying above the prescribed polygonal is generated in Fig. \ref{figabvl}(a). We construct  rational cubic spline FIF in Fig. \ref{figabvl}(b) by changing  the scaling vector $ \alpha$  with respect to the parameters  of Fig. \ref{figabvl}(a).  By taking the null scale vector, i.e., $\alpha=(0,0,0,0)$, we retrieve the classical rational cubic spline plotted in Fig.  \ref{figabvl}(c).
\begin{figure}
\centering
\begin{minipage}{0.3\textwidth}
\epsfig{file=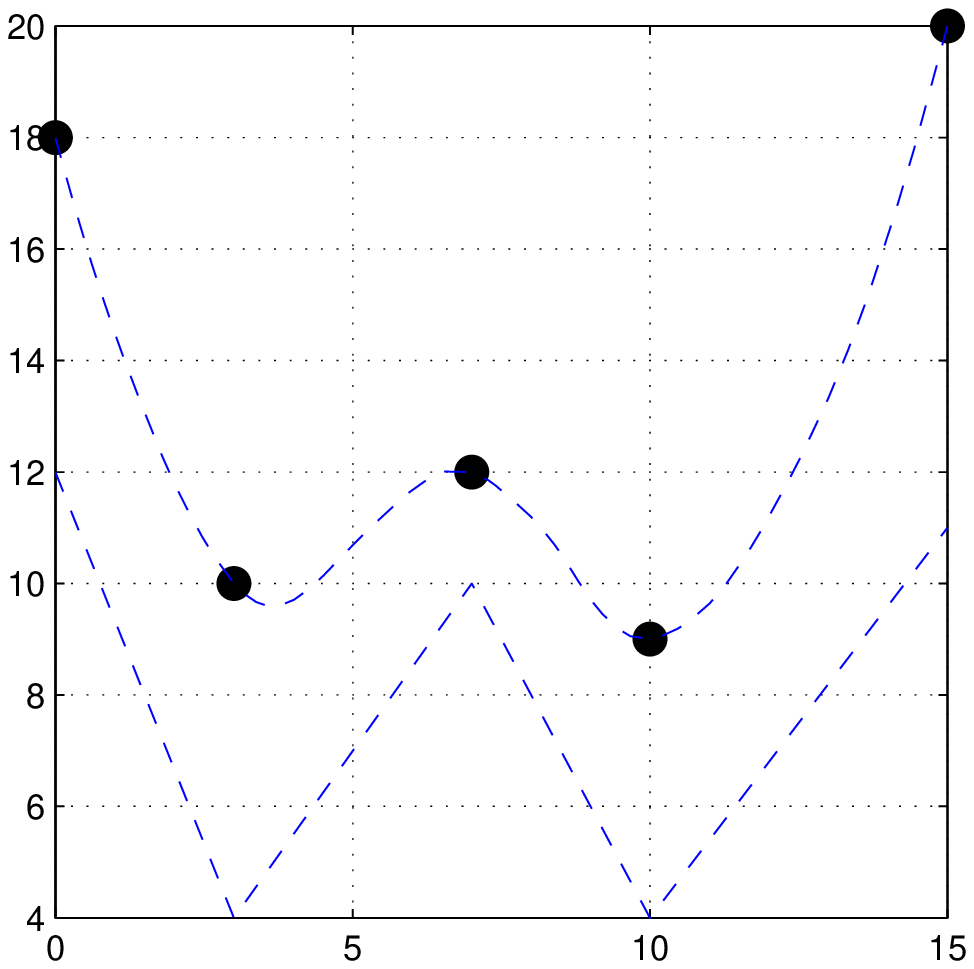,scale=0.3}\\ 
\centering{ (a) Rational cubic spline FIF.}
\end{minipage}\hfill
\begin{minipage}{0.3\textwidth}
\epsfig{file=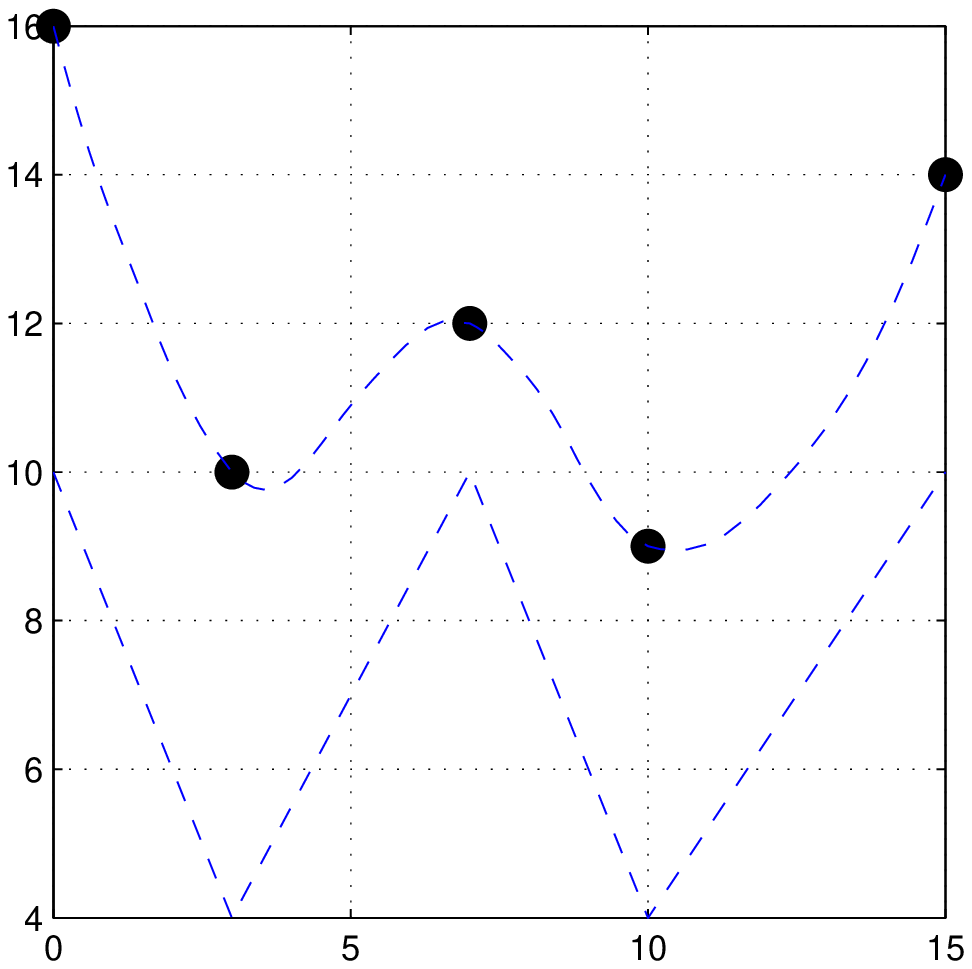,scale=0.3}\\  
 \centering{ (b) Effect of change in $\alpha$ in Fig.1(a). }
 \end{minipage}\hfill
  \begin{minipage}{0.3\textwidth}
\epsfig{file=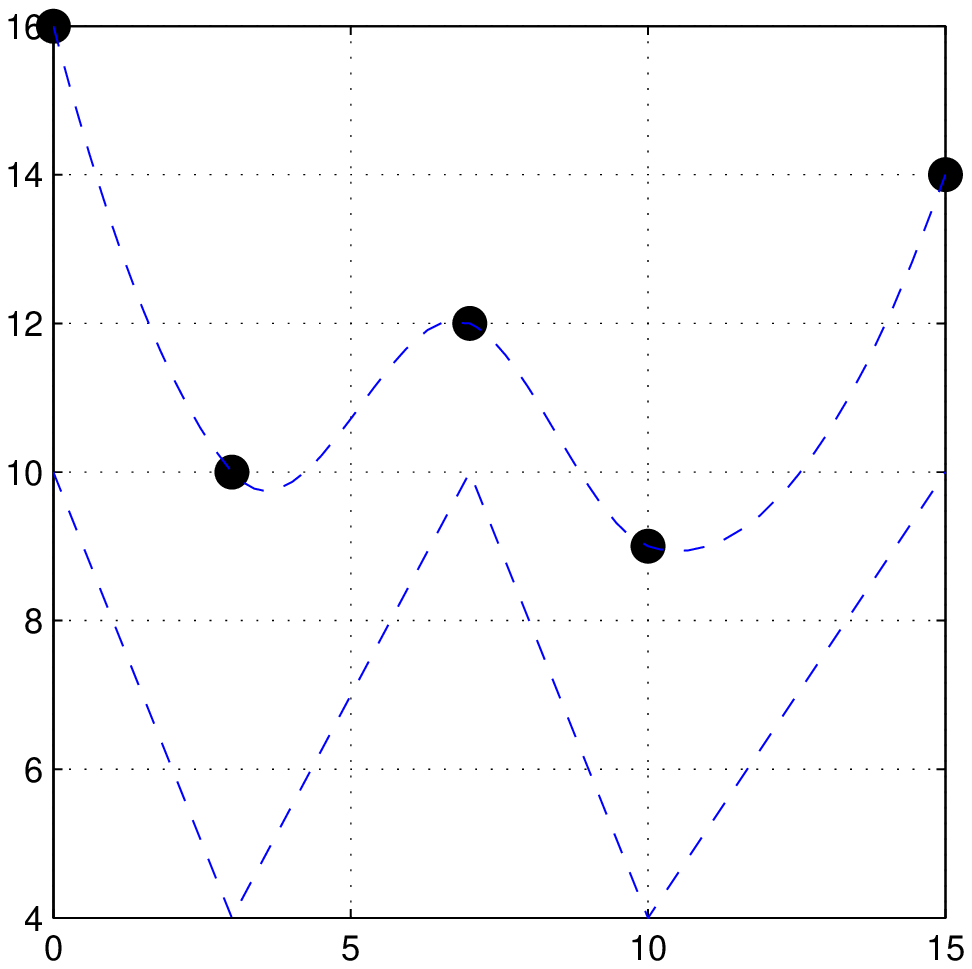,scale=0.3} \\  
\centering{ (c) Classical rational cubic spline FIF.}
\end{minipage}\hfill
 \caption{Rational cubic spline FIFs lying above a piecewise linear function $p$.}
\label{figabvl}
\end{figure}
\\With specific  choices of parameters (see Table \ref{table:Data}) satisfying conditions prescribed in Theorem \ref{fifabvqf}, we obtain constrained rational cubic FIFs displayed in Figs. \ref{figabvq}(a)-(c) lying above the quadratic spline $q$ defined as follows.
$$q(x) =
\begin{cases}
\frac {x^2-7x+20}{2}, & \text{if } \hspace{0.1cm}0\leq x \leq 3 \\
\frac {(x-3)^2-(x-3)+8}{2}, & \text{if } \hspace{0.1cm} 3 \leq x \leq 7 \\
\frac {-(x-7)^2+21(x-7)+60}{6}, & \text{if } \hspace{0.1cm} 7\leq x \leq 10 \\
\frac {87(x-10)^2-375(x-10)+200}{50}, & \text{if } \hspace{0.1cm} 10 \leq x \leq 15.
\end{cases}
$$
\begin{figure}
\centering
\begin{minipage}{0.3\textwidth}
\epsfig{file=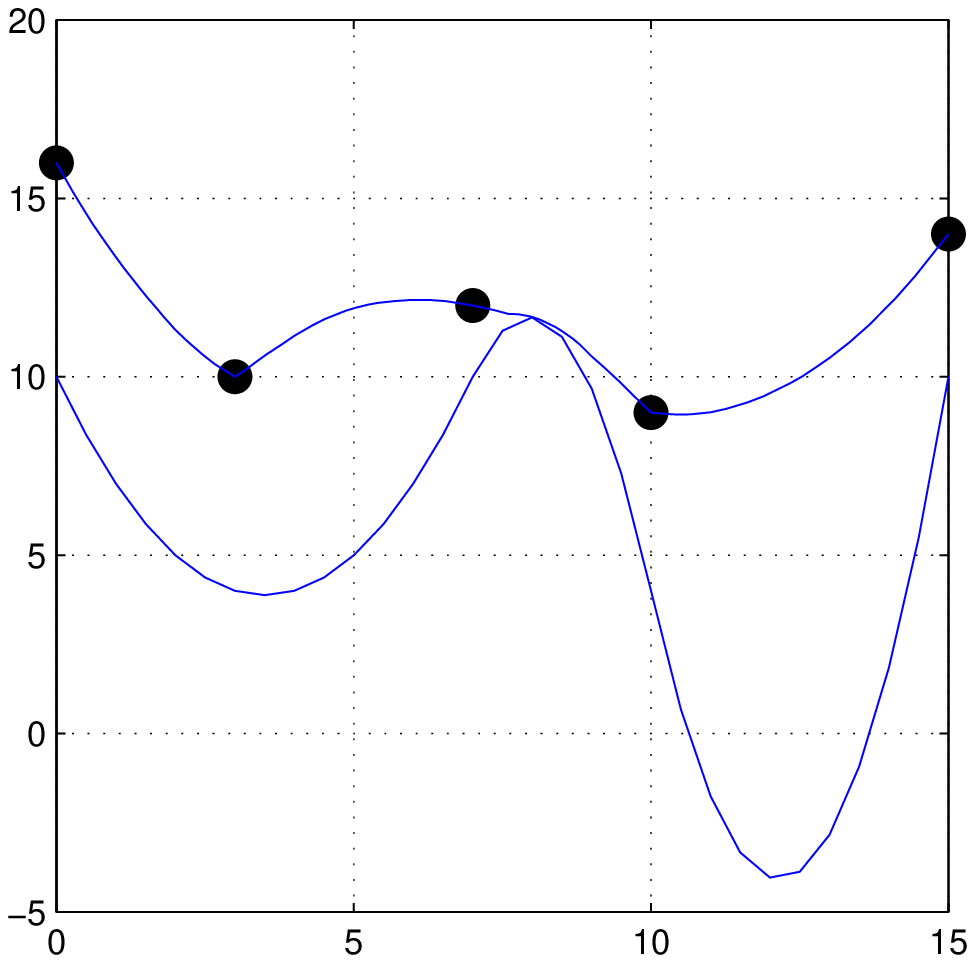,scale=0.3}\\ 
\centering{ (a) Rational cubic spline FIF.}
\end{minipage}\hfill
\begin{minipage}{0.3\textwidth}
\epsfig{file=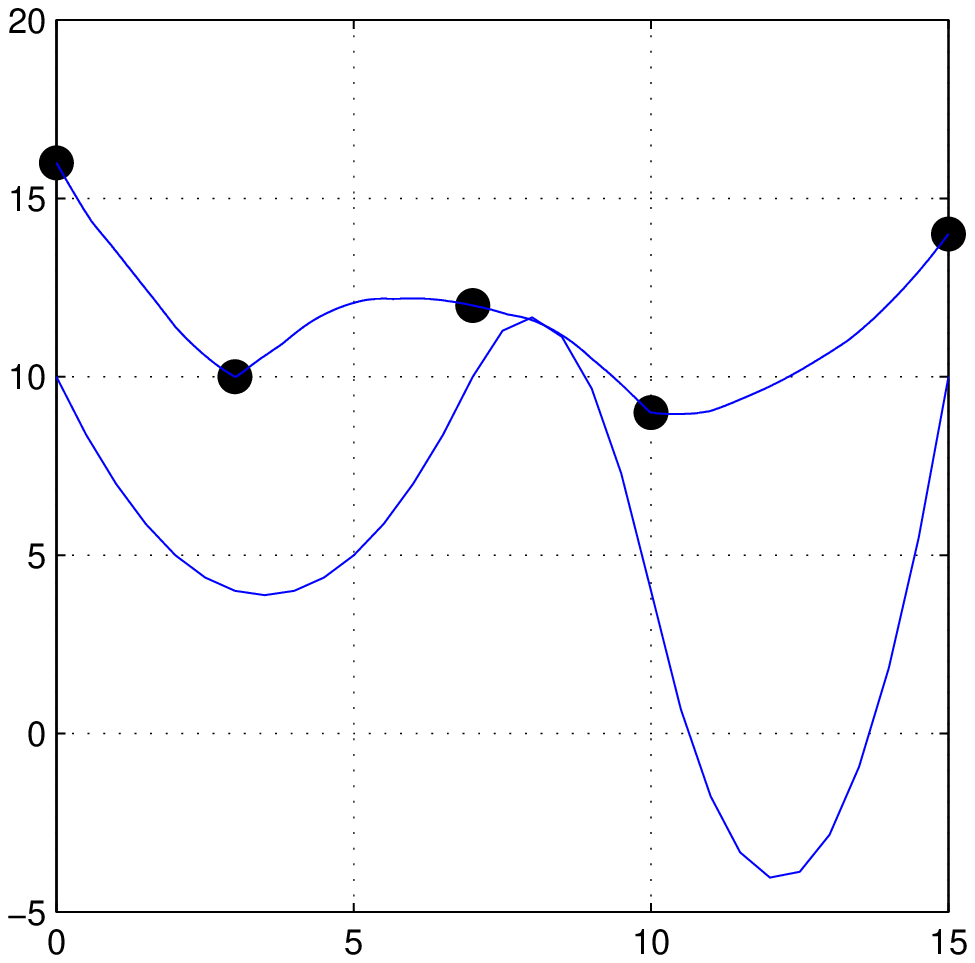,scale=0.3}\\
 \centering{ (b) Effect of change in $\alpha$ in Fig. \ref{figabvq}(a). }
 \end{minipage}\hfill
 \begin{minipage}{0.3\textwidth}
 \epsfig{file=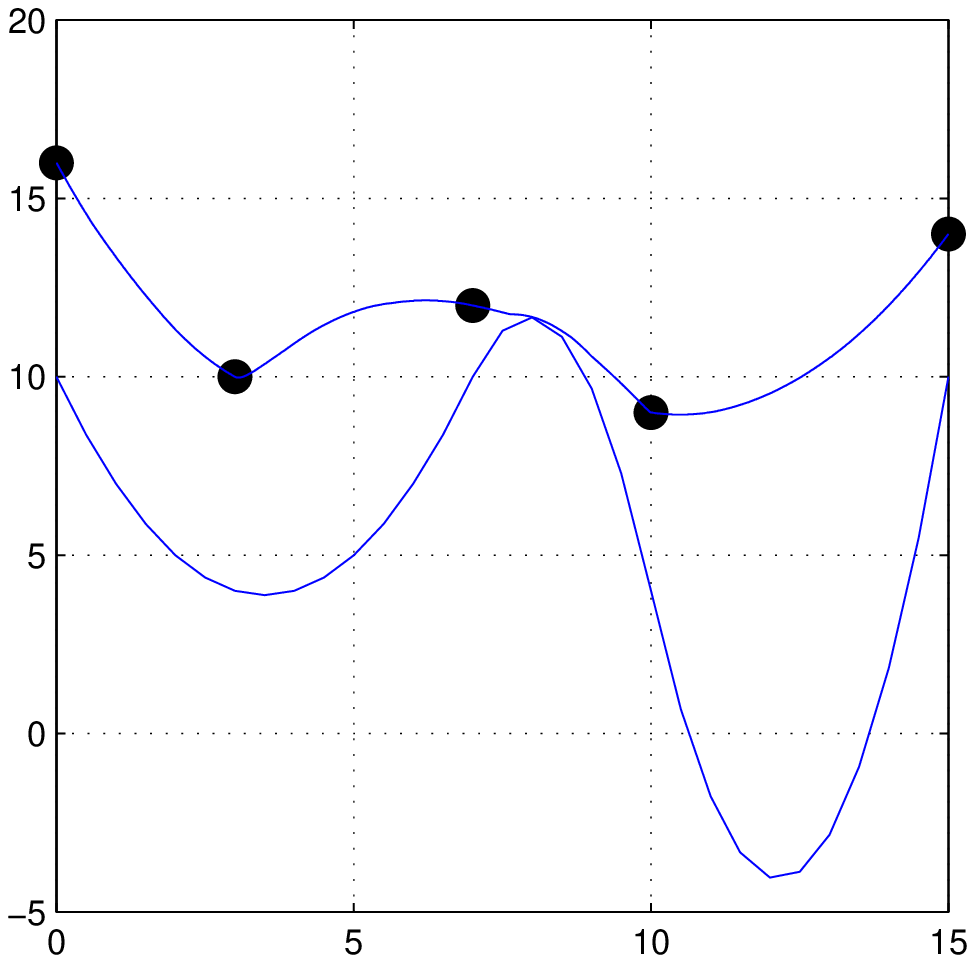,scale=0.3}\\ 
\centering{ (c) Effect of change in shape parameters  in Fig. \ref{figabvq}(a).}
\end{minipage}\hfill
 \caption{Rational cubic spline FIFs lying above a quadratic spline $q$.}
\label{figabvq}
\end{figure}
In contrast to the classical rational spline $f=f^0$, the rational cubic spline FIF $f^\alpha$ has derivative $(f^\alpha)'$ having nondifferentiability in a finite or dense subset of the interpolation interval. Further, the  irregularity may be quantified in terms of  box-counting dimension, which depends mainly on the scaling vector $\alpha$. This may find potential applications in various nonlinear and nonequilibrium phenomena wherein smooth interpolant satisfying a restraint (for instance, positivity, or  more generally lying above or below a prescribed curve) and possessing  irregularity in the derivative of suitable order is required.\\

\noindent \textbf{Note:} Parts of the results of this  paper  were presented  by the third author at the International Conference on Applications of Fractals and Wavelets (ICAFW-2015) held at Amrita School of Engineering, Amrita Vishwa Vidyapeetham, Coimbatore, Tamilnadu, India.

\end{document}